\documentclass[12pt,a4paper]{article}

\PassOptionsToPackage{colorlinks=true,linkcolor=blue,citecolor=blue,urlcolor=blue}{hyperref}
\usepackage[margin=1.1in]{geometry}
\usepackage{amsmath,amssymb,amsthm}
\usepackage{xcolor}
\usepackage{graphicx}
\usepackage{float}
\usepackage{tikz}

\usetikzlibrary{backgrounds,intersections,calc,arrows,arrows.meta,shapes,shapes.geometric,shapes.misc,bbox}
\usepackage[colorlinks=true,linkcolor=blue,citecolor=blue,urlcolor=blue]{hyperref}
\usepackage[abbrev,msc-links,backrefs,nobysame]{amsrefs}
\usepackage{doi}
\usepackage[open,openlevel=2,atend]{bookmark}
\usepackage[T1]{fontenc}
\usepackage{libertinus}
\usepackage{newtxmath}

\hypersetup{colorlinks=true,linkcolor=blue,citecolor=blue,urlcolor=blue}

\renewcommand{\PrintDOI}[1]{\doi{#1}}

\definecolor{xwhite}{RGB}{245,245,245}
\definecolor{xback}{RGB}{0,0,0}
\definecolor{xblue}{RGB}{0,0,139}
\definecolor{xgreen}{RGB}{50,205,50}
\definecolor{xcrimson}{RGB}{220,20,60}
\definecolor{xgold}{RGB}{255,215,0}
\definecolor{xmoccasin}{RGB}{255,228,181}

\newcommand{\vertexFiveStar}{%
  \begin{tikzpicture}[baseline=-0.5ex, scale=0.1]
    \draw (0,0) circle (1);
    \filldraw (0,0) circle (0.12);
    \foreach \angle in {90,162,234,306,378} {
      \draw (0,0) -- ({cos(\angle)}, {sin(\angle)});
      \filldraw ({cos(\angle)}, {sin(\angle)}) circle (0.08);
    }
  \end{tikzpicture}%
}

\pgfdeclarelayer{edgelayer}
\pgfdeclarelayer{nodelayer}
\pgfsetlayers{background,edgelayer,nodelayer,main}

\tikzstyle{none}=[inner sep=0mm]
\tikzstyle{blacknode}=[fill=black, draw=black, shape=circle, minimum size=0.20cm, inner sep=0pt]
\tikzstyle{whitenode}=[fill={rgb,255: red,245; green,245; blue,245}, draw=black, shape=circle, minimum size=0.15cm, inner sep=1pt]
\tikzstyle{whitenode_v1}=[fill={rgb,255: red,245; green,245; blue,245}, draw=black, shape=circle, minimum size=0.4cm, inner sep=0.8pt, scale=1]
\tikzstyle{blacknode_v1}=[fill=black, draw=black, shape=circle, minimum size=0.1cm, inner sep=0pt]
\tikzstyle{rednode}=[fill=red, draw=red, shape=circle, minimum size=0.20cm, inner sep=0pt]
\tikzstyle{bluenode}=[fill=blue, draw=blue, shape=circle, minimum size=0.20cm, inner sep=0pt]
\tikzstyle{black_bold}=[-, draw=black, line width=0.6mm]
\tikzstyle{blackedge}=[-, draw=black, fill=none, line width=0.3mm]
\tikzstyle{rededge_thick}=[-, line width=0.45mm, draw=red]
\tikzstyle{blueedge}=[-, line width=0.3mm, draw=blue]
\tikzstyle{grayedge}=[-,line width=0.2mm, draw={rgb,255: red,64; green,64; blue,64}]

\theoremstyle{plain}
\newtheorem{thm}{Theorem}[section]
\newtheorem{prop}{Proposition}[section]
\newtheorem{lem}{Lemma}[section]

\newtheorem{prob}{Problem}[section]

\newtheorem{remark}{Remark}[section]

\title{\bfseries A note on optimal 2-planar graphs}

\author{
Licheng Zhang\textsuperscript{a}
\and
Yuanqiu Huang\textsuperscript{a}\thanks{Corresponding author.}
\and
Zhangdong Ouyang\textsuperscript{b}
}

\date{\small
\textsuperscript{a}School of Mathematics, Hunan Normal University, Changsha 410082, P.R. China.\\
\textsuperscript{b}Department of Mathematics, Hunan First Normal University, Changsha 410205, P.R. China.\\
E-mail addresses: \texttt{lczhangmath@163.com}, \texttt{hyqq@hunnu.edu.cn}, \texttt{oymath@163.com}.
}

\begin{document}

\maketitle

\begin{abstract}
In this note, we prove that every \(4\)-connected optimal \(2\)-planar graph is Hamiltonian-connected.
Furthermore, we show that the 4-connectedness condition is sharp by constructing infinitely many \(3\)-connected optimal \(2\)-planar graphs that are non-Hamiltonian.
\end{abstract}


\medskip
%\noindent\textbf{MSC 2020:}  05C38; 05C40; 05C62.

\section{Introduction}\label{introsec}
All graphs considered are simple. 
A graph is \emph{\(k\)-planar} if it admits a drawing in the plane in which each edge is crossed at most \(k\) times; the case \(k=2\) gives the class of \emph{\(2\)-planar graphs}.
By a result of Pach and T\'{o}th~\cite{MR1606052}, every \(2\)-planar graph on \(n\ge 3\) vertices has at most \(5n-10\) edges.
A \(2\)-planar graph on \(n\) vertices is called \emph{optimal} if it has exactly \(5n-10\) edges.
%Optimal \(2\)-planar graphs have a rigid structure.
 Several structural and algorithmic aspects of optimal \(2\)-planar graphs have been investigated, such as edge decompositions \cite{MR3895013,MR3685688} and recognition algorithms \cite{MR4424763}.
 
Recall that a graph is \emph{Hamiltonian} if it contains a cycle through all vertices, and is \emph{Hamiltonian-connected} if every two distinct vertices are the endvertices of a Hamiltonian path.
Whitney proved that every \(4\)-connected triangulation is Hamiltonian~\cite{MR1503003}, and Tutte extended this to all \(4\)-connected planar graphs~\cite{MR0081471}. 
These results naturally lead to the question of whether sufficiently high connectivity forces Hamiltonicity in \(k\)-planar graph classes. 
Some results are already known for special \(1\)-planar graphs. For example, Hud\'{a}k, Madaras and Suzuki~\cite{MR2993519} proved that every optimal \(1\)-planar graph is Hamiltonian; Noguchi and Suzuki~\cite{MR3417207} later gave another proof. 
The purpose of this note is to establish a corresponding result for optimal \(2\)-planar graphs.

\begin{thm}\label{thm:main}
Every \(4\)-connected optimal \(2\)-planar graph is Hamiltonian-connected. 
Moreover, there exist infinitely many \(3\)-connected optimal \(2\)-planar graphs that are non-Hamiltonian.
\end{thm}
The positive part is proved by reducing the problem to a planar Hamiltonicity theorem of Thomassen. 
Starting from a suitable drawing of \(G\), we augment its planar skeleton by adding one vertex inside each face, obtaining a triangulation. 
We show that this triangulation is \(4\)-connected, apply the Hamiltonian-connectedness theorem for \(4\)-connected triangulations, and then shortcut the added vertices to get Hamiltonian paths in \(G\). 
For the sharpness part, we construct a family of \(3\)-connected examples by inserting a fixed pentagonal patch into the faces of arbitrary triangulations; after adding pentagrams, the resulting optimal \(2\)-planar graphs fail the standard component condition for Hamiltonicity.

%Theorem~\ref{thm:1} also shows that the Hamiltonicity behavior of optimal \(2\)-planar graphs is closer to that of optimal \(1\)-planar graphs than to arbitrary \(2\)-planar graphs.
%Indeed, the optimality assumption cannot be omitted: there exist \(4\)-connected, even \(5\)-connected, \(1\)-planar graphs that are not Hamiltonian~\cite{MR4130388}.

\section{Face-stellation of a  3-connected plane graph}
In this section, we give  some basic properties of face-stellations, which is a key auxiliary graph for proving Theorem \ref{thm:main}.
The \emph{face-stellation} of a 2-connected plane graph \(G\), denoted by \(G^{\vertexFiveStar}\), is obtained by adding one new vertex inside each face and joining it to all vertices on the boundary of that face. 
Since every face boundary of a \(2\)-connected plane graph is a cycle, \(G^{\vertexFiveStar}\) is a triangulation. 
We call the added vertices \emph{stellating vertices} and the vertices of \(G\) \emph{original vertices}. 
Figure~\ref{fig:face-stellation} shows the face-stellation of the dodecahedral graph.

\begin{figure}[H]
\centering
\begin{tikzpicture}[scale=0.2, bezier bounding box]

\begin{pgfonlayer}{nodelayer}
    % original vertices of the planar skeleton
    \node [style=whitenode] (1) at (6.27737, 2.03974) {};
    \node [style=whitenode] (2) at (-4.59854, -1.4923) {};
    \node [style=whitenode] (3) at (0, 6.60178) {};
    \node [style=whitenode] (4) at (3.88078, -5.34063) {};
    \node [style=whitenode] (5) at (-6.27737, 2.03974) {};
    \node [style=whitenode] (6) at (-3.88078, -5.34063) {};
    \node [style=whitenode] (7) at (0, 10) {};
    \node [style=whitenode] (8) at (5.87591, -8.09002) {};
    \node [style=whitenode] (9) at (1.22466, 1.68694) {};
    \node [style=whitenode] (10) at (1.98297, -0.644769) {};
    \node [style=whitenode] (11) at (-9.50933, 3.09002) {};
    \node [style=whitenode] (12) at (-5.87997, -8.09002) {};
    \node [style=whitenode] (13) at (-1.98297, -0.644769) {};
    \node [style=whitenode] (14) at (2.83861, 3.91322) {};
    \node [style=whitenode] (15) at (4.59448, -1.49635) {};
    \node [style=whitenode] (16) at (9.50933, 3.09002) {};
    \node [style=whitenode] (17) at (-1.22466, 1.68694) {};
    \node [style=whitenode] (18) at (0, -2.08435) {};
    \node [style=whitenode] (19) at (-2.83861, 3.91322) {};
    \node [style=whitenode] (20) at (0, -4.83374) {};

    % stellating vertices, one for each face
    \node [style=bluenode] (h1)  at (3.38, 1.10) {};
    \node [style=bluenode] (h2)  at (3.73, 5.13) {};
    \node [style=bluenode] (h3)  at (6.03, -1.96) {};
    \node [style=bluenode] (h4)  at (-6.03, -1.96) {};
    \node [style=bluenode] (h5)  at (-3.38, 1.10) {};
    \node [style=bluenode] (h6)  at (-2.09, -2.88) {};
    \node [style=bluenode] (h7)  at (-3.73, 5.13) {};
    \node [style=bluenode] (h8)  at (0, 3.56) {};
    \node [style=bluenode] (h9)  at (0, -6.34) {};
    \node [style=bluenode] (h10) at (2.09, -2.88) {};
    \node [style=bluenode] (h11) at (0, 13.0) {};
    \node [style=bluenode] (h12) at (0, 0) {};
\end{pgfonlayer}

\begin{pgfonlayer}{edgelayer}
    % face-stellation edges
    \foreach \v in {1,14,9,10,15}
        \draw [style=blueedge,densely dashed] (h1) to (\v);
    \foreach \v in {14,1,16,7,3}
        \draw [style=blueedge,densely dashed] (h2) to (\v);
    \foreach \v in {1,15,4,8,16}
        \draw [style=blueedge,densely dashed] (h3) to (\v);
    \foreach \v in {2,5,11,12,6}
        \draw [style=blueedge,densely dashed] (h4) to (\v);
    \foreach \v in {5,2,13,17,19}
        \draw [style=blueedge,densely dashed] (h5) to (\v);
    \foreach \v in {2,6,20,18,13}
        \draw [style=blueedge,densely dashed] (h6) to (\v);
    \foreach \v in {3,7,11,5,19}
        \draw [style=blueedge,densely dashed] (h7) to (\v);
    \foreach \v in {14,3,19,17,9}
        \draw [style=blueedge,densely dashed] (h8) to (\v);
    \foreach \v in {8,4,20,6,12}
        \draw [style=blueedge,densely dashed] (h9) to (\v);
    \foreach \v in {4,15,10,18,20}
        \draw [style=blueedge,densely dashed] (h10) to (\v);
    \foreach \v in {9,10,18,13,17}
        \draw [style=blueedge,densely dashed] (h12) to (\v);

    % outer face stellating vertex
    \draw [style=blueedge,densely dashed] (h11) to (7);
    \draw [style=blueedge,densely dashed] (h11) to (11);
    \draw [style=blueedge,densely dashed] (h11) to (16);
    \draw [style=blueedge,densely dashed,out=180,in=120,looseness=1.45] (h11) to (12);
    \draw [style=blueedge,densely dashed,out=0,in=60,looseness=1.45] (h11) to (8);

    % planar skeleton edges
    \draw [style=blackedge] (1) to (14);
    \draw [style=blackedge] (1) to (15);
    \draw [style=blackedge] (1) to (16);
    \draw [style=blackedge] (2) to (5);
    \draw [style=blackedge] (2) to (6);
    \draw [style=blackedge] (2) to (13);
    \draw [style=blackedge] (3) to (7);
    \draw [style=blackedge] (3) to (14);
    \draw [style=blackedge] (3) to (19);
    \draw [style=blackedge] (4) to (8);
    \draw [style=blackedge] (4) to (15);
    \draw [style=blackedge] (4) to (20);
    \draw [style=blackedge] (5) to (11);
    \draw [style=blackedge] (5) to (19);
    \draw [style=blackedge] (6) to (12);
    \draw [style=blackedge] (6) to (20);
    \draw [style=blackedge] (7) to (11);
    \draw [style=blackedge] (7) to (16);
    \draw [style=blackedge] (8) to (12);
    \draw [style=blackedge] (8) to (16);
    \draw [style=blackedge] (9) to (10);
    \draw [style=blackedge] (9) to (14);
    \draw [style=blackedge] (9) to (17);
    \draw [style=blackedge] (10) to (15);
    \draw [style=blackedge] (10) to (18);
    \draw [style=blackedge] (11) to (12);
    \draw [style=blackedge] (13) to (17);
    \draw [style=blackedge] (13) to (18);
    \draw [style=blackedge] (17) to (19);
    \draw [style=blackedge] (18) to (20);
\end{pgfonlayer}

\end{tikzpicture}
\caption{The face-stellation of the dodecahedral graph.}
\label{fig:face-stellation}
\end{figure}
%\begin{lem}\label{lem:planegraphcon}
%
%\end{lem}
%Proposition 7.5

A \emph{chord} of a cycle \(C\) in a graph \(G\) is an edge of \(G\) joining two non-consecutive vertices of \(C\).
\begin{lem}[Tutte, see page 113  in \cite{MR4769863}]\label{lem:tutte}
Let $G$ be a 3-connected plane graph. If $C$ is a  cycle  of $G$ that bounds a face in $G$, then $C$ has no chords.
\end{lem}
 For $S$ be a vertex subset of $G$, let \(G[S]\) denote the subgraph of \(G\) induced by \(S\).   
\begin{lem}\label{lem:stellating-vertex}
Let \(G\) be a \(3\)-connected plane graph, and let \(G^{\vertexFiveStar}\) be its face-stellation. 
Let \(x\) be a stellating vertex inserted into a face \(f\) of \(G\), and let \(C\) be the cycle bounding \(f\). 
Then
$
N_{G^{\vertexFiveStar}}(x)=V(C),
$
and \(G^{\vertexFiveStar}[N_{G^{\vertexFiveStar}}[x]]\) is the wheel obtained from \(C\) by adding \(x\) and joining \(x\) to every vertex of \(C\).
Consequently, if \(y,z\in V(C)\), then \(xyz\) bounds a triangular face of \(G^{\vertexFiveStar}\) when \(y\) and \(z\) are consecutive on \(C\), while \(G^{\vertexFiveStar}[\{x,y,z\}]\) is an induced path with center \(x\) when \(y\) and \(z\) are non-consecutive on \(C\).
\end{lem}

\begin{proof}
By the definition of face-stellation, \(x\) is adjacent precisely to the vertices on the boundary of \(f\). 
Since \(G\) is \(3\)-connected, the boundary \(C\) of \(f\) is an induced cycle by Lemma~\ref{lem:tutte}. 
Thus the neighbors of \(x\) induce exactly the cycle \(C\), and so \(G^{\vertexFiveStar}[N_{G^{\vertexFiveStar}}[x]]\) is a wheel with hub \(x\) and rim \(C\).

The final statement follows immediately: two consecutive vertices of \(C\) are adjacent and form a triangular face with \(x\), while two non-consecutive vertices of \(C\) are nonadjacent and hence induce only the path \(y x z\).
\end{proof}

The \emph{girth} \( g(G) \) of a graph \( G \) is the length of the shortest cycle in \( G \). If \( G \) contains no cycles, then \( g(G) = \infty \).
\begin{prop}\label{prop:star-extension-4conn}
Let \(G\) be a \(3\)-connected plane graph. If \(g(G)\ge 4\), then \(G^{\vertexFiveStar}\) is \(4\)-connected.
\end{prop}

\begin{proof}
Clearly, \(G^{\vertexFiveStar}\) is a triangulation. 
It is well known that a triangulation is \(4\)-connected if and only if it has no separating \(3\)-cycle~\cite{MR0676717}. 
Suppose, to the contrary, that \(G^{\vertexFiveStar}\) has a separating \(3\)-cycle \(T=xyzx\).

By Lemma~\ref{lem:stellating-vertex}, \(T\) contains at most one stellating vertex. 
If \(T\) contains a stellating vertex, then Lemma~\ref{lem:stellating-vertex} implies that \(T\) bounds a triangular face of \(G^{\vertexFiveStar}\), contradicting that \(T\) is separating. 
Hence \(x,y,z\) are all vertices of \(G\). 
But then \(xyzx\) is a \(3\)-cycle of \(G\), contradicting \(g(G)\ge 4\). 
Therefore \(G^{\vertexFiveStar}\) has no separating \(3\)-cycle, and so it is \(4\)-connected.
\end{proof}

%\begin{lem}\label{lem:opplanarskeleton}
%Let $G$ be an optimal 2-planar graph with an OP-drawing $D$. If $\kappa(G) \ge 4$, the planar skeleton $P(G)$ contains no separating 3-cycle.
%\end{lem}

\section{Proof of Theorem \ref{thm:main}}

Let \(D\) be a \(2\)-planar drawing of a graph \(G\).
By deleting all crossed edges from \(D\), we obtain a plane spanning subgraph of \(G\), called the \emph{planar skeleton} of \(D\) and denoted by \(P(D)\).
A plane graph is called a \emph{pentagulation} if every face is bounded by a cycle of length five.

\begin{lem}\label{mainlem}
Let $G$ be a (simple) optimal 2-planar graph. Then $G$ has a 2-planar drawing obtained by drawing  a pentagram (that is, five mutually crossed edges) in the interior of each face of a 3-connected pentagulation.
\end{lem}

 Now, we define such a drawing $D$ of an optimal 2-planar graph \(G\),  which is obtained by inserting a pentagram inside each face of a 3-connected pentagulation, as the \emph{op-drawing} of \(G\).  For example, Figure~\ref{fig1} illustrates an optimal $2$-planar graph on $20$ vertices. 
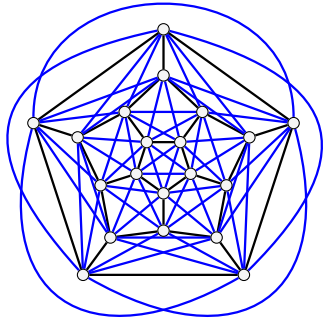
\begin{figure}[H]
\centering
\begin{tikzpicture}[scale=0.18, bezier bounding box]

	\begin{pgfonlayer}{nodelayer}
		\node [style=whitenode] (1) at (6.27737, 2.03974) {};
		\node [style=whitenode] (2) at (-4.59854, -1.4923) {};
		\node [style=whitenode] (3) at (0, 6.60178) {};
		\node [style=whitenode] (4) at (3.88078, -5.34063) {};
		\node [style=whitenode] (5) at (-6.27737, 2.03974) {};
		\node [style=whitenode] (6) at (-3.88078, -5.34063) {};
		\node [style=whitenode] (7) at (0, 10) {};
		\node [style=whitenode] (8) at (5.87591, -8.09002) {};
		\node [style=whitenode] (9) at (1.22466, 1.68694) {};
		\node [style=whitenode] (10) at (1.98297, -0.644769) {};
		\node [style=whitenode] (11) at (-9.50933, 3.09002) {};
		\node [style=whitenode] (12) at (-5.87997, -8.09002) {};
		\node [style=whitenode] (13) at (-1.98297, -0.644769) {};
		\node [style=whitenode] (14) at (2.83861, 3.91322) {};
		\node [style=whitenode] (15) at (4.59448, -1.49635) {};
		\node [style=whitenode] (16) at (9.50933, 3.09002) {};
		\node [style=whitenode] (17) at (-1.22466, 1.68694) {};
		\node [style=whitenode] (18) at (0, -2.08435) {};
		\node [style=whitenode] (19) at (-2.83861, 3.91322) {};
		\node [style=whitenode] (20) at (0, -4.83374) {};
	\end{pgfonlayer}
	\begin{pgfonlayer}{edgelayer}
		\draw [style=blackedge] (1) to (14);
		\draw [style=blackedge] (1) to (15);
		\draw [style=blackedge] (1) to (16);
		\draw [style=blackedge] (2) to (5);
		\draw [style=blackedge] (2) to (6);
		\draw [style=blackedge] (2) to (13);
		\draw [style=blackedge] (3) to (7);
		\draw [style=blackedge] (3) to (14);
		\draw [style=blackedge] (3) to (19);
		\draw [style=blackedge] (4) to (8);
		\draw [style=blackedge] (4) to (15);
		\draw [style=blackedge] (4) to (20);
		\draw [style=blackedge] (5) to (11);
		\draw [style=blackedge] (5) to (19);
		\draw [style=blackedge] (6) to (12);
		\draw [style=blackedge] (6) to (20);
		\draw [style=blackedge] (7) to (11);
		\draw [style=blackedge] (7) to (16);
		\draw [style=blackedge] (8) to (12);
		\draw [style=blackedge] (8) to (16);
		\draw [style=blackedge] (9) to (10);
		\draw [style=blackedge] (9) to (14);
		\draw [style=blackedge] (9) to (17);
		\draw [style=blackedge] (10) to (15);
		\draw [style=blackedge] (10) to (18);
		\draw [style=blackedge] (11) to (12);
		\draw [style=blackedge] (13) to (17);
		\draw [style=blackedge] (13) to (18);
		\draw [style=blackedge] (17) to (19);
		\draw [style=blackedge] (18) to (20);
		\draw [style=blueedge] (19) to (14);
		\draw [style=blueedge] (14) to (17);
		\draw [style=blueedge] (17) to (3);
		\draw [style=blueedge] (3) to (9);
		\draw [style=blueedge] (9) to (19);
		\draw [style=blueedge] (19) to (13);
		\draw [style=blueedge] (13) to (5);
		\draw [style=blueedge] (5) to (17);
		\draw [style=blueedge] (17) to (2);
		\draw [style=blueedge] (2) to (20);
		\draw [style=blueedge] (18) to (6);
		\draw [style=blueedge] (6) to (13);
		\draw [style=blueedge] (20) to (13);
		\draw [style=blueedge] (2) to (18);
		\draw [style=blueedge] (13) to (9);
		\draw [style=blueedge] (9) to (18);
		\draw [style=blueedge] (18) to (17);
		\draw [style=blueedge] (17) to (10);
		\draw [style=blueedge] (10) to (13);
		\draw [style=blueedge] (9) to (1);
		\draw [style=blueedge] (1) to (10);
		\draw [style=blueedge] (10) to (14);
		\draw [style=blueedge] (14) to (15);
		\draw [style=blueedge] (15) to (9);
		\draw [style=blueedge] (18) to (4);
		\draw [style=blueedge] (4) to (10);
		\draw [style=blueedge] (10) to (20);
		\draw [style=blueedge] (20) to (15);
		\draw [style=blueedge] (15) to (18);
		\draw [style=blueedge] (12) to (20);
		\draw [style=blueedge] (20) to (8);
		\draw [style=blueedge] (8) to (6);
		\draw [style=blueedge] (6) to (4);
		\draw [style=blueedge] (4) to (12);
		\draw [style=blueedge] (12) to (5);
		\draw [style=blueedge] (5) to (6);
		\draw [style=blueedge] (2) to (12);
		\draw [style=blueedge] (6) to (11);
		\draw [style=blueedge] (11) to (2);
		\draw [style=blueedge] (11) to (3);
		\draw [style=blueedge] (7) to (19);
		\draw [style=blueedge] (19) to (11);
		\draw [style=blueedge] (5) to (7);
		\draw [style=blueedge] (7) to (1);
		\draw [style=blueedge] (1) to (3);
		\draw [style=blueedge] (3) to (16);
		\draw [style=blueedge] (16) to (14);
		\draw [style=blueedge] (14) to (7);
		\draw [style=blueedge] (1) to (8);
		\draw [style=blueedge] (8) to (15);
		\draw [style=blueedge] (4) to (16);
		\draw [style=blueedge] (16) to (15);
		\draw [style=blueedge] (1) to (4);
		\draw [style=blueedge, bend right=60, looseness=1.75] (7) to (12);
		\draw [style=blueedge, in=-75, out=-30, looseness=1.75] (12) to (16);
		\draw [style=blueedge, in=615, out=-150, looseness=1.75] (8) to (11);
		\draw [style=blueedge, in=-15, out=45, looseness=1.75] (8) to (7);
		\draw [style=blueedge, in=90, out=90, looseness=1.50] (11) to (16);
		\draw [style=blueedge] (19) to (2);
		\draw [style=blueedge] (5) to (3);
	\end{pgfonlayer}
\end{tikzpicture}

\caption{An optimal 2-planar graph of order 20}
\label{fig1}
\end{figure}

\begin{remark}\label{rem:optimal}
In an op-drawing of an optimal \(2\)-planar graph, the boundary vertices of each pentagonal face of the planar skeleton induce a \(K_5\).
\end{remark}

\begin{lem}\label{lem:vertexFiveStar4-connected}
Let \(G\) be an optimal \(2\)-planar graph with an op-drawing \(D\). 
If \(\kappa(G)\ge 4\), then \(P(D)^{\vertexFiveStar}\) is \(4\)-connected.
\end{lem}

\begin{proof}
By Lemma~\ref{mainlem}, \(P(D)\) is a \(3\)-connected pentagulation. 
We claim that \(P(D)\) has no separating \(3\)-cycle. 
Indeed, such a cycle would give a \(3\)-vertex-cut of \(G\), since its three uncrossed edges separate the vertices inside from those outside. 
Thus \(g(P(D))\ge 4\). 
The result now follows from Proposition~\ref{prop:star-extension-4conn}.
\end{proof}

\begin{lem}[Thomassen \cite{MR0698698}]\label{lem:thomassen-planar-4conn}
Let $G$ be a planar graph. If $\kappa(G) \ge 4$, then $G$ is Hamiltonian-connected.
\end{lem}

\begin{proof}[Proof of Theorem~\ref{thm:main}]
First assume that \(G\) is \(4\)-connected, and let \(D\) be an op-drawing of \(G\).
By Lemma~\ref{lem:vertexFiveStar4-connected}, \(P(D)^{\vertexFiveStar}\) is \(4\)-connected. 
Hence, by Lemma~\ref{lem:thomassen-planar-4conn}, \(P(D)^{\vertexFiveStar}\) is Hamiltonian-connected.
That is, for any two vertices  \(x,y\in V(G)\), 
there is a Hamiltonian path \(L\) in \(P(D)^{\vertexFiveStar}\) with endvertices \(x\) and \(y\).
For each stellating vertex \(h\) on \(L\), let \(a\) and \(b\) be its two neighbors on \(L\). 
By the definition of face-stellation, \(a\) and \(b\) lie on the boundary of the same pentagonal face of \(P(D)\). 
By Remark~\ref{rem:optimal}, \(ab\in E(G)\). 
Replacing every subpath \(ahb\) of \(L\) by the edge \(ab\), we obtain a Hamiltonian path between $x$ and $y$ in \(G\). 
Thus \(G\) is Hamiltonian-connected.

We now construct the promised \(3\)-connected non-Hamiltonian examples. The construction is described in four steps.

\begin{enumerate}
\item Let \(H\) be the plane graph shown in Figure~\ref{gadget}. 
From the drawing, \(H\) is \(3\)-connected, \(H-\{x,y,z\}\) is connected, and all faces of \(H\) except the outer face \(xyzx\) are pentagons.

\item Let \(H'\) be a simple triangulation on \(\ell\ge 5\) vertices. 
Then \(H'\) is \(3\)-connected and has \(2\ell-4\) faces. 
Replace each triangular face of \(H'\) by a copy of \(H\), identifying \(xyzx\) with the boundary triangle of the face. 
Let \(H''\) be the resulting plane graph. 
Then \(H''\) is a pentagulation with \(\kappa(H'')=3\).

\item Insert a pentagram inside each pentagonal face of \(H''\). 
This gives an optimal \(2\)-planar graph \(G\). 
Since \(H''\) is a spanning subgraph of \(G\), \(\kappa(G)\ge3\); and the separating triangles inherited from \(H'\) remain \(3\)-vertex-cuts in \(G\). 
Thus \(\kappa(G)=3\).

\item Finally, set \(S=V(H')\). 
Then \(G-S\) has \(2\ell-4\) connected components, while \(|S|=\ell\). 
This contradicts the standard necessary condition for Hamiltonian graphs that
\(c(G-S)\le |S|\) for every vertex set \(S\), where \(c(G-S)\) denotes the number of connected components of \(G-S\); see Bondy and Murty~\cite[Theorem~18.1]{MR2368647}.
Hence \(G\) is non-Hamiltonian.
\end{enumerate}

Since \(\ell\) is arbitrary, this gives infinitely many \(3\)-connected optimal \(2\)-planar graphs that are non-Hamiltonian.
\end{proof}

\vspace{10pt}  

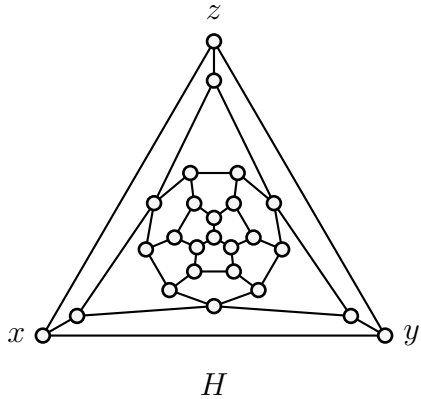
\begin{figure}[H]
\centering

\begin{tikzpicture}[scale=1.3,
  xnode/.style={black,fill=xgreen,line width=1pt},
  ynode/.style={black,fill=xwhite,line width=1pt},
  znode/.style={black,fill=xcrimson,line width=1pt},
  xline/.style={xback,line width=0.8pt},
  yline/.style={xmoccasin,line width=0.8pt}]
  \def\a{20mm}      % 控制尺寸
\def\r{0.7mm}     % 控制节点半径

  \begin{scope}
    \coordinate (v) at (0,0);
    \foreach \i in {1,2,3}{\coordinate (x\i) at (-30+120*\i:\a);}
    \foreach \i in {1,2,3}{\coordinate (y\i) at (-30+120*\i:\a*0.8);}
    \foreach \i in {1,...,9}{\coordinate (z\i) at (-50+40*\i:\a*0.35);}
    \foreach \i in {1,...,6}{\coordinate (u\i) at (-60+60*\i:\a*0.2);}
    \foreach \i in {1,2,3}{\coordinate (v\i) at (-30+120*\i:\a*0.1);}

    \draw[xline] (x1) -- (x2) -- (x3) -- (x1);
    \foreach \i in {1,2,3}{\draw[xline] (x\i) -- (y\i);}
    \foreach \i [remember=\i as \j (initially 9)] in {1,2,...,9}{\draw[xline] (z\i) -- (z\j);}
    \draw[xline] (y1) -- (z2) (y2) -- (z5) (y3) -- (z8);
    \draw[xline] (y1) -- (z5) (y2) -- (z8) (y3) -- (z2);

    \draw[xline] (u1) -- (z1) (u2) -- (z3) (u3) -- (z4) (u4) -- (z6) (u5) -- (z7) (u6) -- (z9);
    \draw[xline] (u1) -- (u2) (u3) -- (u4) (u5) -- (u6);
    \draw[xline] (v1) -- (u2) (v1) -- (u3) (v2) -- (u4) (v2) -- (u5) (v3) -- (u6) (v3) -- (u1);

    \draw[xline] (v) -- (v1) (v) -- (v2) (v) -- (v3);

    \draw[ynode] (v) circle (\r);
    \foreach \i in {1,...,9}{\draw[ynode] (z\i) circle (\r);}
    \foreach \i in {1,...,6}{\draw[ynode] (u\i) circle (\r);}
    \foreach \i in {1,2,3}{
      \draw[ynode] (x\i) circle (\r);
      \draw[ynode] (y\i) circle (\r);
      \draw[ynode] (v\i) circle (\r);
    }

  \end{scope}
  \node  at (0,-1.5) {$H$}; 
  \node  at (-2,-1) {$x$}; 
  \node  at (2,-1) {$y$}; 
  \node  at (0,2.3) {$z$}; 
\end{tikzpicture}

\caption{A plane graph $H$ where all faces except the outer triangular face are pentagons}
\label{gadget}
\end{figure}

\begin{remark}
The graph $H$ was constructed by Parcly Taxel (Jeremy Tan Jie Rui) in response to a question of Prof. Kabenyuk \cite{Kabenyuk:2023:pentagons,Kabenyuk:2024:MinimalPentagulations}.
\end{remark}
\section{Further problems}

We close with a few questions. 
By Theorem~\ref{thm:main}, a non-Hamiltonian optimal \(2\)-planar graph must have connectivity exactly \(3\). 
Our construction has at least \(115\) vertices, so the following extremal question remains open.

\begin{prob}
Determine the minimum order of a non-Hamiltonian optimal \(2\)-planar graph.
\end{prob}

The optimality assumption is also essential in its present form. 
Indeed, Theorem~\ref{thm:main} does not extend to drawing-saturated \(2\)-plane graphs: \(K_{4,6}\) has a saturated \(2\)-planar drawing, but it is \(4\)-connected and non-Hamiltonian; see Figure~\ref{K46}. 
\begin{figure} \centering \begin{tikzpicture}[scale=0.8, bezier bounding box] \begin{pgfonlayer}{nodelayer} \node [style=whitenode] (0) at (3.75, 0.006) {}; \node [style=whitenode] (1) at (-0.005, -1.9877) {}; \node [style=bluenode] (2) at (-1.042, -1.027) {}; \node [style=bluenode] (3) at (-3.55, 0.018) {}; \node [style=bluenode] (4) at (0.982, -0.979) {}; \node [style=bluenode] (5) at (-0.022, 0.47) {}; \node [style=bluenode] (6) at (-0.03, 2.094) {}; \node [style=bluenode] (7) at (2.139, -0.016) {}; \node [style=whitenode] (8) at (-1.06, -0.041) {}; \node [style=whitenode] (9) at (0.984, 0.008) {}; \node [style=none] (10) at (-2.75, 1.25) {}; \end{pgfonlayer} \begin{pgfonlayer}{edgelayer} \draw [style=blackedge, in=90, out=90, looseness=1.50] (0) to (3); \draw [style=blackedge, bend right] (0) to (6); \draw [style=blackedge] (0) to (7); \draw [style=blackedge, in=-75, out=-165] (1) to (3); \draw [style=blackedge] (2) to (1); \draw [style=blackedge] (2) to (8); \draw [style=blackedge] (4) to (1); \draw [style=blackedge] (8) to (5); \draw [style=blackedge, bend left=15] (8) to (6); \draw [style=blackedge] (9) to (4); \draw [style=blackedge] (9) to (5); \draw [style=blackedge] (9) to (7); \draw [style=blackedge,bend right=45] (4) to (0); \draw [style=blackedge, bend right] (1) to (7); \draw [style=blackedge, in=-90, out=-30, looseness=2.25] (9) to (3); \draw [style=blackedge, bend right=75, looseness=2.00] (6) to (1); \draw [style=blackedge](8) to (3); \draw [style=blackedge,in=420, out=105] (0) to (10.center); \draw [style=blackedge,in=-120, out=195, looseness=1.25] (2) to (10.center); \draw [style=blackedge,bend right](4) to (8); \draw [style=blackedge,bend right, looseness=1.25] (2) to (9); \draw [style=blackedge,bend right] (5) to (1); \draw [style=blackedge,bend right=45] (0) to (5); \draw [style=blackedge,bend right=15] (9) to (6); \draw [style=blackedge,bend left=60, looseness=1.50] (8) to (7); \end{pgfonlayer} \end{tikzpicture} \caption{A saturated 2-planar drawing of $K_{4,6}$} \label{K46} \end{figure}
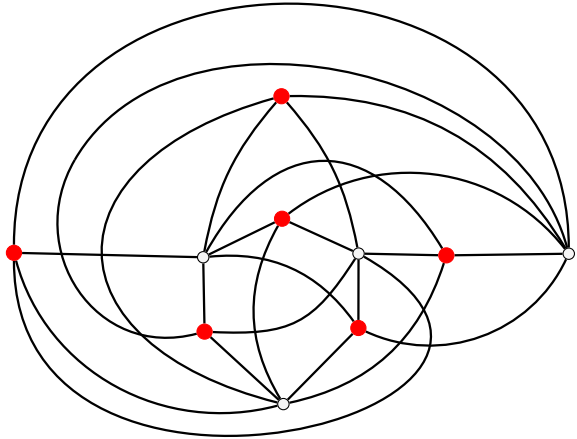
It is therefore natural to ask whether the result remains true for \(4\)-connected maximal \(2\)-planar graph Hamiltonian. A \(2\)-planar graph \(G\) is \emph{maximal \(2\)-planar} if, for every pair of non-adjacent vertices \(u,v\), the graph \(G+uv\) is not \(2\)-planar. 
This should be distinguished from drawing-saturation, which refers to a fixed \(2\)-planar drawing.

\begin{prob}
Is every \(4\)-connected maximal \(2\)-planar graph Hamiltonian?
\end{prob}

More generally, we propose the following problem.
\begin{prob}
For \(2\)-planar graphs, what is the least integer \(\ell\), if it exists, such that every \(\ell\)-connected \(2\)-planar graph is Hamiltonian? 
More generally, does sufficiently high connectivity force Hamiltonicity in \(k\)-planar graphs?
\end{prob}

The corresponding question is already a well-known open problem for \(1\)-planar graphs: it is not known whether every \(6\)-connected \(1\)-planar graph is Hamiltonian; see, for example, Fabrici et al.~\cite{MR4130388}.

\section*{Acknowledgements}

The authors declare that they have no conflicts of interest. This research was supported by grants from the National Natural Science Foundation of China (Grant Nos. 12271157, 12371346).  No data was used for the research described in the article.

%\bibitem{zhang2006}
%Zhang P, Chartrand G. \textit{Introduction to Graph Theory}. New York: Tata McGraw-Hill, 2006.
%\end{thebibliography}
%\bib{MR0100850}{article}{
%   author={Berge, Claude},
%   title={Sur le couplage maximum d'un graphe},
%   language={French},
%   journal={C. R. Acad. Sci. Paris},
%   volume={247},
%   date={1958},
%   pages={258--259},
%   issn={0001-4036},
%   review={\MR{0100850}},
%}

\end{document}